\documentclass{article}
\usepackage{indentfirst}
\usepackage{cite}
\usepackage{amsmath,amssymb,amsthm}
\usepackage{titlesec,hyperref}
\usepackage{color}
\usepackage{fancyhdr}
\usepackage[margin=2.5cm]{geometry}
\usepackage{mathrsfs}
\newtheorem{theo}{Theorem}[section]
\newtheorem{lemm}[theo]{Lemma}

\newtheorem{coro}[theo]{Corollary}
\newtheorem{prop}[theo]{Proposition}
\newtheorem{rema}[theo]{Remark}
\newtheorem{prob}[theo]{Problem}
\numberwithin{equation}{section}

\allowdisplaybreaks

\begin{document}
\title{\textbf{A new boundary of the mapping class group}\footnote{The work was partially supported by NSFC, No: 11771456.}}
\author{\textbf{Lixin $\mbox{Liu}^{1}$ and Yaozhong $\mbox{Shi}^{2,}$}\footnote{Corresponding author.}
\\\vspace{-2mm}{\normalsize $^{1}$Sun Yat-sen University, School of Mathematics}
\\{\normalsize 510275, Guangzhou, P. R. China; mcsllx@mail.sysu.edu.cn}
\\\vspace{-2mm}{\normalsize $^{2}$Sun Yat-sen University, School of Mathematics}
\\{\normalsize 510275, Guangzhou, P. R. China; shiyaozhong314@126.com}}
\date{}
\maketitle
\begin{abstract}
Based on the action of the mapping class group on the space of measured foliations, we construct a new boundary of the mapping class group and study the structure of this boundary. As an application, for any point in Teichm\"uller space, we consider the orbit of this point under the action of the mapping class group and describe the closure of this orbit in the Thurston compactification and the Gardiner-Masur compactification of Teichm\"uller space. We also construct some new points in the Gardiner-Masur boundary of Teichm\"{u}ller space.
\end{abstract}
\noindent \textit{Keywords}:  Mapping class group, Measured foliation, Teichm\"uller space
\\\noindent \textit{MSC (2010)}: 30F60, 32G15, 57M99
\section{Introduction}
In order to study the structure of a group $G$, it is natural to equip $G$ with a boundary. For example, considering the Cayley graph of a finite generated group and rescaling the lengths of the edges by a summable function, we obtain a compact completion of the graph under this new metric. The boundary of this completion is the Floyd boundary of $G$ (see \cite{Floyd}). Besides, a probability measure $\nu$ on $G$ determines a random walk on $G$. There is also a boundary determined by this random walk, which is called the Poisson boundary. The Poisson boundary is strongly related to the harmonic functions corresponding to the random walk (see \cite{Kai}).

Let $S$ be an oriented surface of genus $g$ with $n$ punctures. We assume that $N=3g-3+n\geq1$. In this paper, we study a special group: the mapping class group $Mod(S)$ of $S$. $Mod(S)$ acts on two spaces: $\mathcal{T}(S)$ and $\mathcal{MF}$, where $\mathcal{T}(S)$ is the Teichm\"uller space of $S$ and $\mathcal{MF}$ is the space of measured foliations on $S$. Let $\mathcal{PMF}=\mathcal{MF}-\{0\}/R_{+}$ be the space of projective measured foliations. Based on the action of $Mod(S)$ on $\mathcal{T}(S)$, $\mathcal{MF}$ and $\mathcal{PMF}$, we can study the structure of $Mod(S)$. The most important result in the study of mapping class group may be the Nielsen-Thurston
classification theorem, which states that every $f\in Mod(S)$ is one of three special types: periodic, reducible, or pseudo-Anosov. The structure of subgroups of $Mod(S)$ was also studied by the action of $Mod(S)$ on $\mathcal{T}(S)$ and $\mathcal{MF}$ (see \cite{MCP}, \cite{Ivanov}, etc.).

Different boundaries of $Mod(S)$ were studied by various people (see \cite{Masur-Minsky}, \cite{K-M}, \cite{Ham}, etc.). In particular, Kaimanovich and Masur (see \cite{K-M}) studied the Poisson boundary of $Mod(S)$. They proved that under some natural conditions, the Poisson boundary of $Mod(S)$ is $\mathcal{PMF}$ equipped with a unique measure. In order to obtain their main result, they considered the action of $Mod(S)$ on $\mathcal{PMF}$ and analyzed the asymptotic behaviour of the action of an infinite sequence $\{f_n\}_{n=1}^{\infty}$ on $\mathcal{PMF}$ (see Subsection 1.5 in \cite{K-M}). Inspired by their idea, we study the asymptotic behaviour of $\{f_n\}_{n=1}^{\infty}$ on $\mathcal{MF}$. In the special case that $f_n=f^{n}$ for a fixed $f\in Mod(S)$, we know that
\\(1) when $f$ is a Dehn Twist determined by a simple closed curve $\alpha$, $\lim_{n\rightarrow\infty}\frac{1}{n}f^{n}(F)=i(\alpha,F)\alpha$ for any $F\in\mathcal{MF}$;
\\(2) when $f$ is a pseudo-Anosov map with $\lambda>1$, $f(F^{u})=\lambda F^{u}$, $f(F^{s})=\lambda^{-1}F^{s}$ and $i(F^{u},F^{s})=1$, $\lim_{n\rightarrow\infty}\lambda^{-n}f^{n}(F)=i(F^{s},F)F^{u}$ for any $F\in\mathcal{MF}$.

A natural generalization of these two classical results is
\begin{prob}\label{problem0}
Is the action of $Mod(S)$ on $\mathcal{MF}$ ``projectively precompact"? That is, for any sequence $f_n\in Mod(S)$, are there a subsequence $f_{n_k}$ and a sequence of positive numbers $t_k$ such that $t_kf_{n_k}:\mathcal{MF}\rightarrow\mathcal{MF}$ converges to some map $f_0:\mathcal{MF}\rightarrow\mathcal{MF}$?
\end{prob}
Note that it is necessary to take a subsequence $f_{n_k}$ and a positive scalar $t_k$ in Problem \ref{problem0}, since without these two operations, a generic sequence $f_n$ is not convergent.

We settle Problem \ref{problem0} by embedding $Mod(S)$ into an appropriate space and constructing a new boundary of $Mod(S)$. For this, we need some notations (see Section \ref{sec2}). Let $\Omega(\mathcal{MF})$ be the set of all homogeneous measurable functions from $\mathcal{MF}$ to $\mathcal{MF}$. Note that $R_{+}$ acts on $\Omega(\mathcal{MF})$ by multiplication. Let $P\Omega(\mathcal{MF})=\Omega(\mathcal{MF})-\{0\}/R_{+}$ be the projective space of $\Omega(\mathcal{MF})$. Endow $\Omega(\mathcal{MF})$ with the topology of pointwise convergence and $P\Omega(\mathcal{MF})$ with the quotient topology. Considering the action of $Mod(S)$ on $\mathcal{MF}$, there is a natural map $I:Mod(S)\rightarrow P\Omega(\mathcal{MF})$ (see Section \ref{sec2}). Up to a finite normal subgroup $ker(I)$, we can identify $Mod(S)$ with its image $I(Mod(S))$, which is denoted by $E$ for simplicity. Thus $Mod(S)$ is nearly embedded into $P\Omega(\mathcal{MF})$ and $P\Omega(\mathcal{MF})$ is the appropriate space for settling Problem \ref{problem0}.

With these notations, we prove that the closure $Cl(E)$ of $E=I(Mod(S))$ in $P\Omega(\mathcal{MF})$ is metrizable and compact (Theorem \ref{main}). Note that this result answers Problem \ref{problem0} completely: by the definition of $P\Omega(\mathcal{MF})$, identifying $Mod(S)$ with $E=I(Mod(S))$, the compactness of $Cl(E)$ means that for any sequence $f_n\in Mod(S)$, there are a subsequence $f_{n_k}$ and a sequence of positive numbers $t_k$ such that $t_kf_{n_k}:\mathcal{MF}\rightarrow\mathcal{MF}$ converges to a map $f_0:\mathcal{MF}\rightarrow\mathcal{MF}$.

Identifying $Mod(S)$ with $E=I(Mod(S))$, $\partial E=Cl(E)-E$ is a boundary of $Mod(S)$. For the structure of $\partial E$, we prove (see Section \ref{sec3})
\begin{itemize}
  \item In $Cl(E)$, $E$ is discrete and $\partial E$ is closed (Proposition \ref{prop1}).
  \item The operations of multiplication and inverse on $Mod(S)$ extend continuously to $Cl(E)$ (Proposition \ref{prop2}, \ref{prop3}, \ref{prop4}). But $Cl(E)$ is not indeed a group (Remark \ref{rema2}).
  \item Any point $p\in \partial E$ can be represented as $[\sum_{i=1}^{k}i(E_i,\cdot)F_i]$, where $\{F_i\}\,(\{E_i\})$ are disjoint measured foliations (Theorem \ref{main2}).
  \item Some special points of $\partial E$ are constructed (Proposition \ref{prop5}, \ref{prop6}, \ref{prop7}). In particular, $\partial E\neq\emptyset$.
\end{itemize}

We also consider the actions of $Mod(S)$ on the Thurston compactification $\mathcal{T}^{Th}(S)$ and the Gardiner-Masur compactification $\mathcal{T}^{GM}(S)$ of $\mathcal{T}(S)$. The Thurston boundary is precisely $\mathcal{PMF}$; while the structure of the Gardiner-Masur boundary $GM$ is complex. See \cite{GM}, \cite{Miya-1}, \cite{Miya-2} and \cite{Walsh2} for more details on the Gardiner-Masur boundary. $Mod(S)$ acts on $\mathcal{T}^{Th}(S)$ and $\mathcal{T}^{GM}(S)$ naturally. Considering the actions of $Mod(S)$ on $\mathcal{T}^{Th}(S)$ and $\mathcal{T}^{GM}(S)$, we have two maps
$$\Pi_{Th}:Mod(S)\times \mathcal{T}^{Th}(S)\rightarrow\mathcal{T}^{Th}(S),\,(f,p)\mapsto f(p)$$
and $$\Pi_{GM}:Mod(S)\times \mathcal{T}^{GM}(S)\rightarrow\mathcal{T}^{GM}(S),\,(f,p)\mapsto f(p).$$
If we endow $Mod(S)$ with the discrete topology, then $\Pi_{Th}$ and $\Pi_{GM}$ are both continuous. Since $Cl(E)=E\bigcup \partial E$ is a completion of $Mod(S)$ in some sense, we extend the domains of $\Pi_{Th}$ and $\Pi_{GM}$ to $Cl(E)\times \mathcal{T}^{Th}(S)$ and $Cl(E)\times \mathcal{T}^{GM}(S)$, respectively (see Theorem \ref{main5} and Remark \ref{rema5}).

As an application, we prove Theorem \ref{main3}, which answers the following problem:
\begin{prob}\label{problem1}
For any $x_0$ in $\mathcal{T}(S)$, considering the orbit $\Gamma(x_0)$ of $x_0$ under the action of $Mod(S)$, how to describe the closure of $\Gamma (x_0)$ in $\mathcal{T}^{Th}(S)$ or $\mathcal{T}^{GM}(S)$?
\end{prob}

Besides, using the new boundary $\partial E$, we construct some new points in the Gardiner-Masur boundary of $\mathcal{T}(S)$ (see Remark \ref{rema7}).

The new boundary $\partial E$ is related to a special boundary of $Mod(S)$. For a base point $x\in\mathcal{T}(S)$, identifying $Mod(S)$ with the orbit $\Gamma(x)$ of $x$ by a map $Mod(S)\ni f\mapsto f(x)\in\Gamma(x)$ and then taking boundary in $\mathcal{T}^{Th}(S)$, we get a boundary of $Mod(S)$ depending upon the base point $x$. Note that this boundary is indeed the whole $\mathcal{PMF}$ (see Theorem \ref{main3}). Thus we may call it the ``Thurston boundary" of $Mod(S)$ with base point $x$. And then the new boundary $\partial E$ covers each ``Thurston boundary" of $Mod(S)$ (see Remark \ref{rema6}).

It may be interesting to study the relations between the new boundary $\partial E$ of $Mod(S)$ and some known boundaries of $Mod(S)$, such as the Floyd boundary, the Poisson boundary, etc. We will study these relations in coming future. Besides, Hamenst\"adt introduced a new boundary of $Mod(S)$ (see Section 8 in \cite{Ham2}). It is also interesting to compare our new boundary with Hamenst\"adt's boundary.

We may consider a more general space than $\mathcal{MF}$, that is, the space of geodesic currents (see \cite{Bonahon}). Note that the space of geodesic currents includes $\mathcal{MF}$ and $\mathcal{T}(S)$. Since the construction of the new boundary $\partial E$ is based on the action of $Mod(S)$ on $\mathcal{MF}$ and $Mod(S)$ also acts continuously on the space of geodesic currents, it is natural to ask the following interesting problem:
\begin{prob}\label{problem2}
If we replace the space of measured foliations by the space of geodesic currents, does the construction of the new boundary work?
\end{prob}

This paper is organized as follows.

In Section \ref{sec1}, we introduce background materials on measured foliations, Teichm\"{u}ller space and the action of the mapping class group. In Section \ref{sec2}, we construct the new boundary of $Mod(S)$. In Section \ref{sec3}, we study the structure of the new boundary. In Section \ref{sec4}, we give some applications of our new boundary.

\section{Preliminaries} \label{sec1}
\subsection{Measured foliations}
Let $\mathcal{S}=\mathcal{S}(S)$ be the set of isotopy classes of essential simple closed curves on $S$. For any $\alpha,\beta$ in $\mathcal{S}$, denote by $i(\alpha,\beta)$ the geometric intersection number between $\alpha$ and $\beta$.

Let $R_{\geq 0}=\{x\in R:x\geq 0\}$ and $R_{+}=\{x\in R:x>0\}$. Let $R_{\geq 0}^{\mathcal{S}}$ be the space of non-negative functions on $\mathcal{S}$ endowed with the topology of pointwise convergence. Denote the set of weighted simple closed curves by $R_{+}\times \mathcal{S}=\{t\cdot\alpha :t>0,\alpha\in \mathcal{S}\}$. It is known that
\begin{equation*}
i_{*}:R_{+}\times \mathcal{S}\rightarrow R_{\geq 0}^{\mathcal{S}},
\end{equation*}
\begin{equation*}
t\cdot \alpha\mapsto t\cdot i(\alpha,\cdot)
\end{equation*}
is injective and induces a topology on $R_{+}\times\mathcal{S}$. With this topology, $i_{*}$ is an embedding.

The closure of $i_{*}(R_{+}\times\mathcal{S})$ in $R_{\geq 0}^{\mathcal{S}}$ is called the space of measured foliations on $S$, which is denote by $\mathcal{MF}$. $R_{+}$ acts on $R_{\geq 0}^{\mathcal{S}}$ by multiplication. Denote $R_{\geq 0}^{\mathcal{S}}-\{0\}/R_{+}$ by $PR_{\geq 0}^{\mathcal{S}}$ and $\mathcal{MF}-\{0\}/R_{+}$ by $\mathcal{PMF}$. $\mathcal{PMF}$ is called the space of projective measured foliations. For $F\in\mathcal{MF}-\{0\}$, denote $[F]\in\mathcal{PMF}$ to be the projective class of $F$. Note that $\mathcal{S}$ is embedded in $PR_{\geq 0}^{\mathcal{S}}$, and the closure of $\mathcal{S}$ in $PR_{\geq 0}^{\mathcal{S}}$ is $\mathcal{PMF}$. It is well known that $\mathcal{MF}$ is homeomorphic to $R^{6g-6+2n}$ and $\mathcal{PMF}$ is homeomorphic to $S^{6g-7+2n}$ (see \cite{FLP}).

For two weighted simple closed curves $t\alpha,s\beta\in R_{+}\times\mathcal{S}$, define their intersection number by the homogeneous equation $i(t\alpha,s\beta)=tsi(\alpha,\beta)$. Then the intersection number function $i$ extends continuously to $i:\mathcal{MF}\times\mathcal{MF}\rightarrow R_{\geq 0}$.

Any $F\in\mathcal{MF}-\{0\}$ is represented by a singular foliation with a transverse measure $\mu$ in the sense that for any simple closed curve $\alpha$,
\begin{equation*}
i(F,\alpha)=\inf_{\alpha^{'}}\int_{\alpha^{'}}d\mu,
\end{equation*}
where the infimum is over all simple closed curves $\alpha^{'}$ homotopic to $\alpha$.

Besides, we need the definition of ergodic decomposition of a measured foliation. A saddle connection of a foliation is a leaf connecting two singularities (not necessarily distinct). The critical graph of a foliation is defined to be the union of all saddle connections. The complement of the critical graph contains finitely many connected components. Each connected component is either a cylinder swept out by closed leaves or a so-called minimal component (each leaf is dense). On every minimal component $D$, there exists a finite set of ergodic transverse measures $\mu_1,...,\mu_n$ such that any transverse measure $\mu$ on $D$ can be written as $\mu=\sum_{j=1}^{n}f_j\mu_j$ for some non-negative coefficients $\{f_j\}$. An indecomposable component of a measured foliation is either a cylinder with a positive weight or a minimal component $D$ with a ergodic measure $\mu_j$. A measured foliation is called indecomposable if it contains only one indecomposable component. With these notations, any measured foliation $F$ can be uniquely represented as
\begin{equation*}
F=\sum_{i=1}^{k}F_i,
\end{equation*}
where each $F_i$ is an indecomposable measured foliation such that $i(F_i,F_j)=0$ and $[F_i]\neq[F_j]$ for $i\neq j$. We call this the ergodic decomposition of $F$.

Finitely many simple closed curves $\alpha_1,\alpha_2,...,\alpha_k$ fill up $S$ if for any $F\in \mathcal{MF}-\{0\}$,
\begin{equation*}
\sum_{i=1}^{k}i(\alpha_i,F)>0.
\end{equation*}

We need the following result (Lemma 6.3, \cite{Walsh1}).
\begin{lemm} \label{lemm0}
Let $\{F_i:i=0,1,2,...,k\}$ be some projectively distinct indecomposable measured foliations such that $i(F_i,F_j)=0\,(i\neq j)$. Then for any $\epsilon>0$, there exists a simple closed curve $\alpha$ such that
\begin{equation*}
i(F_i,\alpha)<\epsilon i(F_0,\alpha),\,i=1,2,...,k.
\end{equation*}
\end{lemm}
\subsection{Teichm\"{u}ller space and its compactifications}
Let $\mathcal{T}(S)$ be the Teichm\"{u}ller space of $S$. There are two equivalent definitions of $\mathcal{T}(S)$: the set of isotopy classes of hyperbolic metrics on $S$ and the set of isotopy classes of conformal structures on $S$. For two hyperbolic metric $m_1,m_2$ of finite area on $S$, $m_1$ is equivalent to $m_2$ if there exists an orientation-preserving homeomorphism $f:S\rightarrow S$ isotopic to the identity map such that $f_{*}m_1=m_2$, where $f_{*}m_1$ is the push-forward of $m_1$ by $f$. $\mathcal{T}(S)$ is defined to be the set of equivalence classes of hyperbolic metrics of finite area on $S$. For two conformal structure $\mu_1,\mu_2$ on $S$, $\mu_1$ is equivalent to $\mu_2$ if there exists an orientation-preserving homeomorphism $f:S\rightarrow S$ isotopic to the identity map such that $f_{*}\mu_1=\mu_2$, where $f_{*}\mu_1$ is the push-forward of $\mu_1$ by $f$. $\mathcal{T}(S)$ is also defined to be the set of equivalence classes of conformal structures on $S$. By the uniformization theorem, these two definitions are consistent.

For the definition corresponding to hyperbolic metric, we consider the hyperbolic length function on $\mathcal{T}(S)$. For any $x\in\mathcal{T}(S)$ and $\alpha\in\mathcal{S}$, let $l(x,\alpha)$ be the hyperbolic length of the geodesic isotopic to $\alpha$ in the hyperbolic metric corresponding to $x$. The hyperbolic length of a simple closed curve extends continuously to the hyperbolic length of a measured foliation. The map
\begin{equation*}
l(\cdot,\cdot):\mathcal{T}(S)\times \mathcal{MF}\rightarrow R,
\end{equation*}
\begin{equation*}
(x,F)\mapsto l(x,F)
\end{equation*}
is continuous.

Thurston constructed a compactification of Teichm\"uller space by the hyperbolic lengths of simple closed curves. Define a map $\widetilde{\varphi}_{Th}$ by
\begin{equation*}
\widetilde{\varphi}_{Th}: \mathcal{T}(S)\rightarrow R_{\geq0}^{\mathcal{S}},
\end{equation*}
\begin{equation*}
x\mapsto (l(x,\alpha))_{\alpha\in\mathcal{S}}.
\end{equation*}

Let $pr:R_{\geq0}^{\mathcal{S}}-\{0\}\rightarrow PR_{\geq0}^{\mathcal{S}}$ be the projective map. Then the map $\varphi_{Th}=pr\circ\widetilde{\varphi}_{Th}$ is an embedding and the closure of the image is compact. Moreover, $Cl(\mathcal{T}(S))-\mathcal{T}(S)=\mathcal{PMF}$. Thus we have a compactification of $\mathcal{T}(S)$ denoted by $\mathcal{T}^{Th}(S)=\mathcal{T}(S)\bigcup\mathcal{PMF}$. $\mathcal{T}^{Th}(S)$ is the Thurston compactification and $\mathcal{PMF}$ is the Thurston boundary.

A sequence $\{x_n\}_{n=1}^{\infty}$ in $\mathcal{T}(S)$ converges to a boundary point $[F]$ in $\mathcal{PMF}$ if and only if there exists a positive sequence $\{t_n\}_{n=1}^{\infty}$ such that $\lim_{n\rightarrow\infty}t_n=0$ and $\lim_{n\rightarrow\infty}t_nl(x_n,\alpha)=i(F,\alpha)$ for any $\alpha\in\mathcal{S}$.

For the definition corresponding to conformal structure, we consider the extremal length function on $\mathcal{T}(S)$. For any $x\in\mathcal{T}(S)$ and $\alpha\in\mathcal{S}$, let $Ext(x,\alpha)$ be the extremal length of $\alpha$ in the conformal structure corresponding to $x$. The extremal length of a simple closed curve extends continuously to the extremal length of a measured foliation. For more details on extremal length, see \cite{Kerc}. The map
\begin{equation*}
Ext(\cdot,\cdot):\mathcal{T}(S)\times \mathcal{MF}\rightarrow R,
\end{equation*}
\begin{equation*}
(x,F)\mapsto Ext(x,F)
\end{equation*}
is continuous.

Gardiner and Masur constructed a compactification of Teichm\"uller space by the extremal lengths of simple closed curves in \cite{GM}. Define a map $\widetilde{\varphi}_{GM}$ by
\begin{equation*}
\widetilde{\varphi}_{GM}: T(S)\rightarrow R_{\geq0}^{\mathcal{S}},
\end{equation*}
\begin{equation*}
x\mapsto (Ext^{\frac{1}{2}}(x,\alpha))_{\alpha\in\mathcal{S}}.
\end{equation*}

The map $\varphi_{GM}=pr\circ\widetilde{\varphi}_{GM}$ is an embedding and the closure of the image is compact. Thus we have a compactification of $\mathcal{T}(S)$ denoted by $\mathcal{T}^{GM}(S)=\mathcal{T}(S)\bigcup GM$. $\mathcal{T}^{GM}(S)$ is the Gardiner-Masur compactification and $GM$ is the Gardiner-Masur boundary.

Different from the Thurston boundary $\mathcal{PMF}$, the structure of Gardiner-Masur boundary $GM$ is much more complex. For more details on its structure, see \cite{Miya-1}, \cite{Miya-2}, \cite{Walsh2}.
\subsection{The action of the mapping class group}
Let $Mod(S)$ be the mapping class group of surface $S$, which is the set of isotopy classes of orientation-preserving homeomorphisms of $S$. $Mod(S)$ acts on $\mathcal{MF}$ and $\mathcal{T}(S)$ by push-forward. Precisely, for $f\in Mod(S)$ and $x\in\mathcal{T}(S)$, if $m$ and $\mu$ are a hyperbolic metric and a conformal structure in the equivalence class $x$, respectively, define $f(x)$ to be the the equivalence class of $f_{*}m$ or $f_{*}\mu$. For a measured foliation $(F,\nu)$, define $f(F,\nu)$ to be $(f(F),f_{\ast}\nu)$. And its action on $\mathcal{T}(S)$ extends naturally to the Thurston compactification and the Gardiner-Masur compactification of $\mathcal{T}(S)$. For more details on $Mod(S)$, see \cite{FM}.

In this paper, we use the following convention: for any $x\in \mathcal{T}(S)$, $f\in Mod(S)$, $F\in \mathcal{MF}$,
\begin{equation*}
l(f(x),F)=l(x,f^{-1}(F)),\,Ext(f(x),F)=Ext(x,f^{-1}(F)).
\end{equation*}
\section{Construction of the new boundary}\label{sec2}
Based on the action of $Mod(S)$ on the measured foliation space $\mathcal{MF}=\mathcal{MF}(S)$, we construct a new boundary of $Mod(S)$ in this section.

For any $f\in Mod(S)$, $f$ acts on $\mathcal{MF}$ as a homogeneous continuous map $f:\mathcal{MF}\rightarrow \mathcal{MF}$, which is measurable in particular. Recall that $f$ is homogeneous if for any $F\in \mathcal{MF}$, $k\geq0$, $f(kF)=kf(F)$. Denote the set of all homogeneous measurable maps from $\mathcal{MF}$ to $\mathcal{MF}$ by $\Omega(\mathcal{MF})$. We endow $\Omega(\mathcal{MF})$ with the topology of pointwise convergence.

Since $R_{+}$ acts on $\mathcal{MF}$ by multiplication, multiplying any $f\in\Omega(\mathcal{MF})$ by a positive number $k$, we get a homogeneous measurable map $kf:\mathcal{MF}\rightarrow\mathcal{MF},\,F\mapsto kf(F)$. Thus $R_{+}$ also acts on $\Omega(\mathcal{MF})$ by multiplication. Then we have the projective space $P\Omega(\mathcal{MF})=\Omega(\mathcal{MF})-\{0\}/R_{+}$, where $0$ is the zero element in $\Omega(\mathcal{MF})$. Let $\pi:\Omega(\mathcal{MF})-\{0\}\rightarrow P\Omega(\mathcal{MF})$ be the projective map. Denote $[f]=\pi(f)$ to be the projective class of $f\in\Omega(\mathcal{MF})$.

We endow $P\Omega(\mathcal{MF})$ with the quotient topology induced by $\pi$. Precisely, for a sequence $\{[f_n]\}_{n=0}^{\infty}$ in $P\Omega(\mathcal{MF})$, $\lim_{n\rightarrow\infty}[f_n]=[f_0]$ if and only if there exists a positive sequence $\{t_n\}_{n=1}^{\infty}$ such that $t_nf_n$ converges to $f_0$ in the topology of pointwise convergence.

Sending $f\in Mod(S)$ to its action $f:\mathcal{MF}\rightarrow\mathcal{MF}$, we have a natural map
\begin{equation*}
\widetilde{I}:Mod(S)\rightarrow \Omega(\mathcal{MF}).
\end{equation*}
Composing it with $\pi$, we have another map
\begin{equation*}
I=\pi\circ\widetilde{I}:Mod(S)\rightarrow P\Omega(\mathcal{MF}).
\end{equation*}

The kernel $ker(I)=\{f\in Mod(S):[f:\mathcal{MF}\rightarrow\mathcal{MF}]=[id_{\mathcal{MF}}]\}$ is finite. In fact, if $f\in ker(I)$, then there exists a positive number $k$ such that for any $F\in\mathcal{MF}$, $kf(F)=F$. Since $R_{+}\mathcal{S}$ is dense in $\mathcal{MF}$, this is equivalent to that for any $\alpha\in \mathcal{S}$, $f(\alpha)=k\alpha$. Since $f$ sends a simple closed curve to a simple closed curve, we have $k=1$. Thus $f\in kerI$ is equivalent to that $f$ fixes the isotopy class of each essential simple closed curve. By the result in page 344 of \cite{FM}, we know that when the topology type $(g,n)$ of $S$ is $(2,0),(1,1),(1,2)$ or $(0,4)$, $ker(I)$ is a subgroup of order $2$ or $4$; in the other cases, $ker(I)$ is trivial. So up to the finite normal subgroup $ker(I)$ (with order $1,2$ or $4$), we can identify $Mod(S)$ with its image $I(Mod(S))$. For simplicity, denote the image $I(Mod(S))$ by $E$.

Let $Cl(E)$ be the closure of $E$ in $P\Omega(\mathcal{MF})$. The main result of this section is
\begin{theo}\label{main}
$Cl(E)$ is metrizable and compact.
\end{theo}
Thus $Cl(E)$ is a completion of $Mod(S)$ and $\partial E=Cl(E)-E$ is a boundary of $Mod(S)$ in some sense.

We need some preparations to prove Theorem \ref{main}. In order to give a clear description to the topology of $\mathcal{MF}$, we choose $N$ simple closed curves $\{\alpha_1,\alpha_2,...,\alpha_N\}$ filling up the surface $S$ such that the map
\begin{equation*}
\Phi: \mathcal{MF}\rightarrow R^{N},
\end{equation*}
\begin{equation*}
F\mapsto \big(i(\alpha_1,F),i(\alpha_2,F),...,i(\alpha_N,F)\big)
\end{equation*}
is an embedding (see \cite{FLP}). As a result, we identify $\mathcal{MF}$ with the image $\Phi(\mathcal{MF})$ which is endowed with the Euclidean metric on $R^{N}$.

Let $l(\cdot)=\sum_{i=1}^{N}i(\alpha_i,\cdot):\mathcal{MF}\rightarrow R_{\geq0}$ be the length function on $\mathcal{MF}$ corresponding to $\{\alpha_1,\alpha_2,...,\alpha_N\}$. Since $\{\alpha_1,\alpha_2,...,\alpha_N\}$ fill up the surface, $l(F)=0$ if and only if $F=0$. Recall a result from \cite{Bonahon}:
\begin{lemm}\label{lemm1}
For any $M>0$, $\{F\in \mathcal{MF}:l(F)\leq M\}$ is compact in $\mathcal{MF}$.
\end{lemm}

From Lemma \ref{lemm1}, we have
\begin{lemm} \label{lemm2}
$\Phi(\mathcal{MF})$ is closed in $R^{N}$.
\end{lemm}
\begin{proof}
Suppose $\lim_{n\rightarrow\infty}\Phi(F_n)=f_0$ for some sequence $\{F_n\}_{n=1}^{\infty}\subseteq\mathcal{MF}$ and $f_0=(a_1,a_2,...,a_N)\in R^N$. Then $\lim_{n\rightarrow\infty}i(\alpha_i,F_n)=a_i$ for $i=1,2,...,N$. So $l(F_n)=\sum_{i=1}^{N}i(\alpha_i,F_n)\leq M$ for some $M>0$. From Lemma \ref{lemm1}, there exists a subsequence $\{F_{n_k}\}_{k=1}^{\infty}$ such that $\lim_{k\rightarrow\infty}F_{n_k}=F_0$ for some $F_0\in\mathcal{MF}$. Since $\Phi$ is continuous, we have $\Phi(F_0)=f_0$. Thus $f_0\in \Phi(\mathcal{MF})$, which completes the proof.
\end{proof}

Let $\Omega^{'}(\mathcal{MF})\subseteq\Omega(\mathcal{MF})$ be the set of homogeneous continuous maps from $\mathcal{MF}$ to $\mathcal{MF}$. And let $P\Omega^{'}(\mathcal{MF})\subseteq P\Omega(\mathcal{MF})$ be the projective space of $\Omega^{'}(\mathcal{MF})$. Since $Mod(S)$ acts continuously on $\mathcal{MF}$, $E\subseteq P\Omega^{'}(\mathcal{MF})$.

Now we proceed to construct a metric on $P\Omega^{'}(\mathcal{MF})$. Set $\mathcal{MF}_1=\{F\in\mathcal{MF}:l(F)\leq 1\}$. For any $[f]$ in $P\Omega^{'}(\mathcal{MF})$, we define the normalized lift of $[f]$ to $\Omega^{'}(\mathcal{MF})$ by
\begin{equation*}
\widehat{f}(\cdot)=\frac{f(\cdot)}{L(f)}:\mathcal{MF}\rightarrow \mathcal{MF},
\end{equation*}
where $L(f)=\sup_{\mathcal{MF}_1}l(f(\cdot))$. Note that $L(f)$ is finite because of the compactness of $\mathcal{MF}_1$.

Let $d$ be the Euclidean metric on $\mathcal{MF}$ induced by $\Phi$: for any $F,G\in\mathcal{MF}$, $d(F,G)=|\Phi(F)-\Phi(G)|$, where $|\cdot|$ is the Euclidean norm on $R^{N}$. We define a map $\widehat{d}:P\Omega^{'}(\mathcal{MF})\times P\Omega^{'}(\mathcal{MF})\rightarrow R$ as follows: for any $[f],[g]\in P\Omega^{'}(\mathcal{MF})$,
\begin{equation*}
\widehat{d}([f],[g])=\sup_{F\in\mathcal{MF}_1}d(\widehat{f}(F),\widehat{g}(F)).
\end{equation*}
Note that $\widehat{d}$ is a metric on $P\Omega^{'}(\mathcal{MF})$. In fact, the symmetry and the triangle inequality come from these two properties of metric $d$; the positive definiteness comes from the definition of the normalized lift $\widehat{f}$.

For the metric $\widehat{d}$, we have
\begin{lemm}\label{lemm3}
For any $\{[f_n]\}_{n=0}^{\infty}\subseteq P\Omega^{'}(\mathcal{MF})$, $\lim_{n\rightarrow \infty}\widehat{d}([f_n],[f_0])=0$ if and only if there exists a positive sequence $\{t_n\}_{n=1}^{\infty}$ such that $t_nf_n$ converges uniformly to $f_0$ on any compact subset of $\mathcal{MF}$.
\end{lemm}
\begin{proof}
	Suppose that $\lim_{n\rightarrow\infty}\widehat{d}([f_n],[f_0])=0$. Then from the definition of metric $\widehat{d}$, we know that $f_n(\cdot)/L(f_n)$ converges uniformly to $f_0(\cdot)/L(f_0)$ on $\mathcal{MF}_1$. For any compact subset $A$ of $\mathcal{MF}$, set $l(A)=\sup_{F\in A}l(F)$. Note that for any $F\in A$, $F/l(A)\in \mathcal{MF}_{1}$. Thus we have $f_n(\cdot)/l(A)L(f_n)$ converges uniformly to $f_0(\cdot)/l(A)L(f_0)$ on $A$, which also implies that $L(f_0)f_n(\cdot)/L(f_n)$ converges uniformly to $f_0(\cdot)$ on $A$.
	
	Conversely, suppose that there exists a sequence $\{t_n\}_{n=1}^{\infty}\subseteq R_{+}$ such that $t_nf_n$ converges uniformly to $f_0$ on any compact subset of $\mathcal{MF}$. In particular, $t_nf_n$ converges uniformly to $f_0$ on $\mathcal{MF}_{1}$, which implies that $\frac{t_nf_n(\cdot)}{L(t_nf_n)}$ converges uniformly to $\frac{f_0(\cdot)}{L(f_0)}$ on $\mathcal{MF}_{1}$. Thus
	\begin{equation*}
		\lim_{n\rightarrow\infty}\hat{d}([f_n],[f_0])=\lim_{n\rightarrow\infty}\sup_{\mathcal{MF}_1}d(\frac{f_n(\cdot)}{L(f_n)},\frac{f_0(\cdot)}{L(f_0)})
		=\lim_{n\rightarrow\infty}\sup_{\mathcal{MF}_1}d(\frac{t_nf_n(\cdot)}{L(t_nf_n)},\frac{f_0(\cdot)}{L(f_0)})=0.
	\end{equation*}
\end{proof}

In the metric space $(P\Omega^{'}(\mathcal{MF}),\widehat{d})$, $E$ is precompact:
\begin{lemm}\label{lemm4}
	For any sequence $\{[f_n]\}_{n=1}^{\infty}\subseteq E$, there exists a subsequence $\{[f_{n_k}]\}_{k=1}^{\infty}$ such that
	$$
	\lim_{k\rightarrow\infty}\widehat{d}([f_{n_k}],[f_0])=0
	$$
	for some $[f_0]\in P\Omega^{'}(\mathcal{MF})$.
\end{lemm}
\begin{proof}
	From the definition of $E$, we assume that $f_n\in Mod(S)$ ($n=1,2,...$). Take a point $x_0\in\mathcal{T}(S)$. Note that the action of $Mod(S)$ on $\mathcal{T}(S)$ is properly discontinuous. Thus by the definition of the Thurston compactification of $\mathcal{T}(S)$, one of the followings holds:
	\\(1) $f_{n_k}\equiv f_0$ for some subsequence $\{f_{n_k}\}_{k=1}^{\infty}$ and $f_0\in Mod(S)$;
	\\(2) $\lim_{k\rightarrow\infty}f_{n_k}(x_0)=[F_0]$ for some subsequence $\{f_{n_k}\}_{k=1}^{\infty}$ and $[F_0]\in \mathcal{PMF}$.
	
	For the case (1), $\widehat{d}([f_{n_k}],[f_0])\equiv 0.$
	
	For the case (2), there exists a sequence of positive numbers $\{t_k\}_{k=1}^{\infty}$ such that for any $F\in\mathcal{MF}$,
	$$
	\lim_{k\rightarrow\infty}t_kl(x_0,f^{-1}_{n_k}(F))=\lim_{k\rightarrow\infty}t_kl(f_{n_k}(x_0),F)=i(F_0,F).
	$$
	In particular, we have $\lim_{k\rightarrow\infty}l(x_0,t_kf^{-1}_{n_k}(\alpha_i))=i(F_0,\alpha_i)$ for $i=1,2,...,N$. Since $\{\alpha_i\}_{i=1}^{N}$ fill up the surface, we have $l(F_0)=\sum_{i=1}^{N}i(\alpha_i,F_0)>0$, which implies that $m\leq \sum_{i=1}^{N}l(x_0,t_kf^{-1}_{n_k}(\alpha_i))\leq M$ for some $m,M>0$. Note that $\{F\in\mathcal{MF}:l(x_0,F)\leq M\}$ is compact. Thus passing to a subsequence again, we assume that $\lim_{k\rightarrow\infty}t_kf^{-1}_{n_k}(\alpha_i)=F_i$ for some $F_i\in\mathcal{MF}$ and $F_{i_0}\neq 0$ for some $i_0$.
	
	By the definition of $\Phi$, we have
	\begin{equation*}
		t_k\Phi\circ f_{n_k}(\cdot)=\big(i(\alpha_i,t_kf_{n_k}\cdot)\big)_{i=1}^{N}=\big(i(t_kf^{-1}_{n_k}(\alpha_i),\cdot)\big)_{i=1}^{N}.
	\end{equation*}
	Since $i(\cdot,\cdot):\mathcal{MF}\times\mathcal{MF}\rightarrow R_{\geq0}$ is continuous and $\lim_{k\rightarrow\infty}t_kf^{-1}_{n_k}(\alpha_i)=F_i$, we know that $t_k\Phi\circ f_{n_k}(\cdot)=\big(i(t_kf^{-1}_{n_k}(\alpha_i),\cdot)\big)_{i=1}^{N}$ converges uniformly to $\big(i(F_i,\cdot)\big)_{i=1}^{N}\neq 0$ on any compact subset of $\mathcal{MF}$. By Lemma \ref{lemm2}, for any $F\in\mathcal{MF}$, $\big(i(F_i,F)\big)_{i=1}^{N}=\lim_{k\rightarrow\infty}\big(i(\alpha_i,t_kf_{n_k}(F))\big)_{i=1}^{N}\in \Phi(\mathcal{MF})$, which implies that $f_0=\Phi^{-1}\big(i(F_i,\cdot)\big)_{i=1}^{N}$ is a homogeneous continuous map from $\mathcal{MF}$ to $\mathcal{MF}$. Since $\Phi$ is a homeomorphism, $t_kf_{n_k}$ converges uniformly to $f_0$ on any compact subset of $\mathcal{MF}$. By Lemma \ref{lemm3}, $\lim_{k\rightarrow\infty}\widehat{d}([f_{n_k}],[f_0])=0$.
\end{proof}

By Lemma \ref{lemm3} and Lemma \ref{lemm4}, we prove Theorem \ref{main} now.
\\\textbf{Proof of Theorem \ref{main}.}
Firstly, we prove that $Cl(E)\subseteq P\Omega^{'}(\mathcal{MF})$. Naturally, $E\subseteq P\Omega^{'}(\mathcal{MF})$. Suppose that a sequence $[f_n]\in E$ converges to $[f_0]\in Cl(E)$ in $P\Omega(\mathcal{MF})$. By Lemma \ref{lemm4}, there exists a subsequence $\{[f_{n_k}]\}_{k=1}^{\infty}$ such that
$\lim_{k\rightarrow\infty}\widehat{d}([f_{n_k}],[f^{'}_0])=0$ for some $[f^{'}_{0}]\in P\Omega^{'}(\mathcal{MF})$. By Lemma \ref{lemm3}, there exists a sequence of positive numbers $\{t_k\}_{k=1}^{\infty}$ such that $t_kf_{n_k}$ converges uniformly to $f^{'}_0$ on any compact subset of $\mathcal{MF}$. Since the uniform convergence on compact sets is stronger than the pointwise convergence, $t_kf_{n_k}$ converges to $f^{'}_0$ in the topology of pointwise convergence.
By the topology of $P\Omega(\mathcal{MF})$, $[f_{n_k}]$ converges to $[f^{'}_{0}]$ in $P\Omega(\mathcal{MF})$, which implies that $[f_0]=[f^{'}_{0}]\in P\Omega^{'}(\mathcal{MF})$. Thus $Cl(E)\subseteq P\Omega^{'}(\mathcal{MF})$.

Secondly, we prove that the topology on $Cl(E)$ is coincident with  that induced by the metric $\widehat{d}$, that is, $Cl(E)$ is metrizable. From Lemma \ref{lemm3} and the fact that the uniform convergence on compact sets is stronger than the pointwise convergence, we know that for any sequence $\{[f_n]\}^{\infty}_{n=0}$ in $Cl(E)$, $\lim_{n\rightarrow\infty}\widehat{d}([f_n],[f_0])$ implies that $[f_n]$ converges to $[f_0]$ in the topology of $Cl(E)$. For the inverse direction, suppose that  $[f_n]$ converges to $[f_0]$ in the topology of $Cl(E)$. We wish to prove that $\lim_{n\rightarrow\infty}\widehat{d}([f_n],[f_0])=0$. We prove this by contradiction. Suppose that there exists a subsequence $\{[f_{n_k}]\}_{k=1}^{\infty}$ such that $\widehat{d}([f_{n_k}],[f_0])>\varepsilon$ for some $\varepsilon>0$. Then by Lemma \ref{lemm4}, passing to a subsequence again, we can assume that $\lim_{k\rightarrow \infty}\widehat{d}([f_{n_k}],[f^{'}_0])=0$ for some $[f^{'}_{0}]\in P\Omega^{'}(\mathcal{MF})$, which implies that $[f^{'}_0]\in Cl(E)$ and  $[f_{n_k}]$ converges to $[f^{'}_{0}]$ in the topology of $Cl(E)$. Thus $[f^{'}_0]=[f_0]$. By
$$
0=\lim_{k\rightarrow \infty}\widehat{d}([f_{n_k}],[f^{'}_0])=\lim_{k\rightarrow \infty}\widehat{d}([f_{n_k}],[f_0])\geq\varepsilon,
$$
we get a contradiction.

Finally, by Lemma \ref{lemm4}, as a dense subset of metric space $(Cl(E),\widehat{d})$, $E$ is precompact. Thus $Cl(E)$ is compact.\qed

By the proof of Theorem \ref{main}, we have two useful corollaries.
\begin{coro}\label{coro1}
The boundary point set $\partial E=Cl(E)-E$ is included in $P\Omega^{'}(\mathcal{MF})$, that is, any boundary point $p$ can be represented by $p=[f_p]$, where $f_p$ is a homogeneous continuous map from $\mathcal{MF}$ to $\mathcal{MF}$.
\end{coro}
\begin{coro}\label{coro2}
For any sequence $\{[f_n]\}_{n=0}^{\infty}$ in $Cl(E)$, the followings are equivalent:
\\(1) $\lim_{n\rightarrow\infty}[f_n]=[f_0]$;
\\(2) there exists a sequence of positive numbers $\{t_n\}_{n=1}^{\infty}$ such that $t_nf_n$ converges to $f_0$ in the topology of pointwise convergence;
\\(3) there exists a sequence of positive numbers $\{t_n\}_{n=1}^{\infty}$ such that $t_nf_n$ converges uniformly to $f_0$ on any compact subset of $\mathcal{MF}$.
\end{coro}
Corollary \ref{coro2} means that in $Cl(E)$, the pointwise convergence and the uniform convergence on compact sets are equivalent, which is not true in general.

\section{The structure of the boundary} \label{sec3}
In this section, we study the structure of the boundary $\partial E=Cl(E)-E$ in details. Recall that we endow $P\Omega(\mathcal{MF})$ with the quotient topology from the pointwise convergence on $\Omega(\mathcal{MF})$. $Cl(E)$ is the closure of $E$ in this topology.

\begin{prop} \label{prop1}
In $Cl(E)$, $E$ is discrete and $\partial E$ is closed.
\end{prop}
\begin{proof}
For any $[f]$ in $E$, if $[f]$ is not an isolated point in $Cl(E)$, then there exists a sequence $\{[f_n]\}_{n=1}^{\infty}$ in $E$ such that $f_n\neq f_m$ for $n\neq m$ and $\lim_{n\rightarrow\infty}[f_n]=[f]$. So there exists a sequence of positive numbers $\{t_n\}_{n=1}^{\infty}$ such that $\lim_{n\rightarrow\infty}t_nf_n=f$. Choosing a point $x_0$ in $\mathcal{T}(S)$, we have
\begin{equation*}
\lim_{n\rightarrow\infty}t_nl(f_{n}^{-1}(x_0),\cdot)=\lim_{n\rightarrow\infty}l(x_0,t_nf_n(\cdot))=l(x_0,f(\cdot))
=l(f^{-1}(x_0),\cdot).
\end{equation*}
Thus $f_n^{-1}(x_0)$ converges to $f^{-1}(x_0)$ in $\mathcal{T}(S)$, which contradicts the properly discontinuity of the action of $Mod(S)$ on $\mathcal{T}(S)$. Thus any point of $E$ is isolated in $Cl(E)$, which implies that $E$ is discrete in $Cl(E)$.

Since $E$ is discrete in $Cl(E)$, $E$ is open in $Cl(E)$. Thus $\partial E$ is closed in $Cl(E)$.
\end{proof}

The operations of multiplication and inverse on $Mod(S)$ extend continuously to $Cl(E)$. For this, we need some notations. Let $\widetilde{E}=\pi^{-1}(E)$ be the inverse image of $E$ in $\Omega(\mathcal{MF})$ and $Cl(\widetilde{E})$ be the closure of $\widetilde{E}$ in $\Omega(\mathcal{MF})$. By Corollary \ref{coro1}, $Cl(\widetilde{E})\subseteq \Omega^{'}(\mathcal{MF})$. Similar to Corollary \ref{coro2}, we have
\begin{coro}\label{coro3}
For any sequence $\{f_n\}_{n=0}^{\infty}$ in $Cl(\widetilde{E})-\{0\}$, the followings are equivalent:
\\(1) $\lim_{n\rightarrow \infty}f_n=f_0$, that is, $f_n$ converges to $f_0$ in the topology of pointwise convergence;
\\(2) $f_n$ converges uniformly to $f_0$ on any compact subsets of $\mathcal{MF}$.
\end{coro}
\begin{proof}
Since the uniform convergence on compact sets is stronger than the pointwise convergence, (2) implies (1).

For the inverse direction, suppose that $f_n$ converges to $f_0$ in the topology of pointwise convergence. Then by Corollary \ref{coro2},  there exists a sequence of positive numbers $\{t_n\}_{n=1}^{\infty}$ such that $t_nf_n$ converges uniformly to $f_0$ on any compact subset of $\mathcal{MF}$. In particular, $t_nf_n$ converges to $f_0$ in the topology of pointwise convergence, which implies that $\lim_{n\rightarrow\infty}t_n=1$. Therefore, $f_n=\frac{t_nf_n}{t_n}$ converges uniformly to $f_0$ on any compact subsets of $\mathcal{MF}$.
\end{proof}
By Corollary \ref{coro3}, we have
\begin{prop} \label{prop2}
	The map $M:Cl(\widetilde{E})-\{0\}\times Cl(\widetilde{E})-\{0\}\rightarrow \Omega^{'}(\mathcal{MF})$ defined by $(f,g)\mapsto f\circ g$ is continuous. And for any $f,g\in Cl(\widetilde{E})-\{0\}$, $f\circ g\in Cl(\widetilde{E})$.
\end{prop}
\begin{proof}
Firstly, we prove the continuity of $M$. Suppose that $\lim_{n\rightarrow \infty}f_n=f$ and $\lim_{n\rightarrow \infty}g_n=g$ in $Cl(\widetilde{E})-\{0\}$. By Corollary \ref{coro3}, $f_n$ and $g_n$ converge uniformly to $f$ and $g$ on any compact subset of $\mathcal{MF}$, respectively. We need to prove that $f_n\circ g_n$ converges to $f\circ g$ in the topology of pointwise convergence.

Let $d$ be the Euclidean metric on $\mathcal{MF}$ induced by $\Phi$. For any $F\in\mathcal{MF}$, since $\lim_{n\rightarrow\infty}g_n(F)=g(F)$, we know that for any $\epsilon>0$, there is $N_1>0$ such that for any $n>N_1$,
$$
d(f\circ g_n(F),f\circ g(F))<\frac{\epsilon}{2}
$$
and $\{g_n(F)\}_{n=1}^{\infty}\subseteq B$ for some compact neighbourhood $B$ of $g(F)$.

Since $f_n$ converges uniformly to $f$ on $B$, there exists $N_2>0$ such that for any $n>N_2$ and $F\in B$, $d\big(f_n(F),f(F)\big)<\frac{\epsilon}{2}$.

Thus for any $n>\max\{N_1,N_2\}$,
\begin{equation*}
d(f_n\circ g_n(F),f\circ g(F))\leq d(f_n\circ g_n(F),f\circ g_n(F))+d(f\circ g_n(F),f\circ g(F))<\frac{\epsilon}{2}+\frac{\epsilon}{2}=\epsilon,
\end{equation*}
which implies that $f_n\circ g_n(F)$ converges to $f\circ g(F)$. Thus $M$ is continuous.

Secondly, we prove that $f,g\in Cl(\widetilde{E})-\{0\}$ implies $f\circ g\in Cl(\widetilde{E})$. Take $f_n,g_n\in Mod(S)$ and $t_n,k_n>0$ such that $\lim_{n\rightarrow\infty}t_nf_n=f$ and $\lim_{n\rightarrow\infty}k_ng_n=g$. Since $M$ is continuous, $\lim_{n\rightarrow\infty}t_nk_nf_n\circ g_n=f\circ g$, which implies that $f\circ g\in Cl(\widetilde{E})$.
\end{proof}

By Proposition \ref{prop2}, for any two element $[f],[g]$ in $Cl(E)$, if $f\circ g\neq 0$, we can define the product of $[f]$ and $[g]$ by $[f\circ g]\in Cl(E)$. In particular, restricting to $Mod(S)$, it is coincident with the multiplication operation on $Mod(S)$. Thus the multiplication operation on $Mod(S)$ extends continuously to $Cl(E)$ except in some degenerated cases ($f\circ g=0$).

For the inverse operation, we have
\begin{prop} \label{prop3}
For any $f$ in $Cl(\widetilde{E})$, there exists a unique element $\overline{f}$ in $Cl(\widetilde{E})$ such that $i\big(f(F),G\big)=i\big(F,\overline{f}(G)\big)$ for any $F,G$ in $\mathcal{MF}$. And the map $\varphi:Cl(\widetilde{E})\rightarrow Cl(\widetilde{E}),\,f\mapsto\overline{f}$ is a homeomorphism.
\end{prop}
\begin{proof}
Firstly, we prove the existence of $\overline{f}$. For any $f$ in $Cl(\widetilde{E})$, if $f=Kf_{0}$ for some $K\geq0$ and $f_0\in Mod(S)$, we set $\overline{f}=Kf^{-1}_0$.

For other cases, we assume that $\lim_{n\rightarrow\infty}t_nf_n=f$ for some $t_n>0,\,f_n\in Mod(S)\,(n=1,2,...)$. Then we have
\begin{equation*}
\lim_{n\rightarrow\infty}\Phi(t_n f_n^{-1}(\cdot))=\lim_{n\rightarrow\infty}\big(i(\alpha_i,t_n f_n^{-1}(\cdot))\big)_{i=1}^{N}=\lim_{n\rightarrow\infty}\big(i(t_n f_n(\alpha_i),\cdot)\big)_{i=1}^{N}=\big(i(f(\alpha_i),\cdot)\big)_{i=1}^{N}.
\end{equation*}
From Lemma \ref{lemm2}, we set $f_{0}=\Phi^{-1}(\big(i(f(\alpha_i),\cdot)\big)_{i=1}^{N})$ and then
\begin{equation*}
\lim_{n\rightarrow\infty}t_nf_n^{-1}=f_0.
\end{equation*}
Thus for any $F,G$ in $\mathcal{MF}$, we have
\begin{equation*}
i(f(F),G)=\lim_{n\rightarrow\infty}i\big(t_nf_n(F),G\big)
=\lim_{n\rightarrow\infty}i\big(F,t_nf_n^{-1}(G)\big)=i\big(F,f_0(G)\big).
\end{equation*}
So we set $\overline{f}=f_0$.

Secondly, we prove the uniqueness of $\overline{f}$. For any $f\in Cl(\widetilde{E})$, suppose there are two elements $f_1,f_2$ such that for any $F,G\in \mathcal{MF}$,
\begin{equation*}
i\big(f(F),G\big)=i\big(F,f_1(G)\big)=i\big(F,f_2(G)\big).
\end{equation*}
Then
\begin{equation*}
\Phi\big(f_1(\cdot)\big)=\Phi\big(f_2(\cdot)\big)=\big(i(f(\alpha_i),\cdot)\big)_{i=1}^{N}.
\end{equation*}
Since $\Phi$ is an embedding, we know that $f_1=f_2$.

Now we prove that $\varphi$ is a homeomorphism. Obviously, we have $\overline{\overline{f}}=f$ for any $f$ in $Cl(\widetilde{E})$, which implies that $\varphi^{2}=id:Cl(\widetilde{E})\rightarrow Cl(\widetilde{E})$. Thus we only need to prove that $\varphi$ is continuous. Suppose $\lim_{n\rightarrow\infty}f_n=f_0$ for $\{f_n\}_{n=0}^{\infty}$ in $Cl(\widetilde{E})$. Then we have
\begin{equation*}
\lim_{n\rightarrow\infty}\Phi(\overline{f_n}(\cdot))
=\lim_{n\rightarrow\infty}\big(i(\alpha_i,\overline{f_n}(\cdot))\big)_{i=1}^{N}
=\lim_{n\rightarrow\infty}\big(i(f_n(\alpha_i),\cdot)\big)_{i=1}^{N}
\end{equation*}
\begin{equation*}
=\big(i(f_0(\alpha_i),\cdot)\big)_{i=1}^{N}
=\big(i(\alpha_i,\overline{f_0}(\cdot))\big)_{i=1}^{N}
=\Phi(\overline{f_0}(\cdot)).
\end{equation*}
Since $\Phi$ is an embedding, $\lim_{n\rightarrow\infty}\overline{f_n}=\overline{f_0}$. Thus $\varphi$ is continuous.
\end{proof}
We call $\overline{f}$ defined in Proposition \ref{prop3} the conjugate of $f$. For any $[f]\in Cl(E)$, define the conjugate of $[f]$ by $[\overline{f}]$. In particular, for any $f\in Mod(S)$, the conjugate of $f$ is exactly the inverse of $f$ in $Mod(S)$. Thus the inverse operation on $Mod(S)$ extends continuously to $Cl(E)$.

There is a natural relation between the operations of multiplication and conjugate on $Cl(\widetilde{E})$.
\begin{prop} \label{prop4}
For any $f,g\in Cl(\widetilde{E})$, $\overline{f\circ g}=\overline{g}\circ\overline{f}$.
\end{prop}
\begin{proof}
From the definition of the conjugate operation, we know that for any $F,G\in\mathcal{MF}$,
\begin{equation*}
i(f\circ g(F),G)=i(g(F),\overline{f}(G))=i(F,\overline{g}\circ\overline{f}(G)),
\end{equation*}
which implies that $\overline{f\circ g}=\overline{g}\circ\overline{f}$.
\end{proof}
\begin{rema} \label{rema2}
By proposition \ref{prop2}, \ref{prop3}, \ref{prop4}, the natural group structure of $Mod(S)$ extends continuously to $Cl(E)$. But $Cl(E)$ is not a group, since there are some degenerated cases that the multiplication is not defined and the conjugate operation on $Cl(E)$ is not indeed an inverse operation for a group.
\end{rema}

Now we prove a lemma:
\begin{lemm}\label{lemm7}
Suppose $\lim_{n\rightarrow\infty}[f_n]=[f_0]$ for some $\{f_n\}_{n=1}^{\infty}$ in $Mod(S)$ and $[f_0]\in \partial E$. Then there exists a sequence of positive numbers $\{t_n\}_{n=1}^{\infty}$ such that $\lim_{n\rightarrow\infty}t_n=0$ and $\lim_{n\rightarrow\infty}t_nf_n=f_0$.
\end{lemm}
\begin{proof}
Since $\lim_{n\rightarrow\infty}[f_n]=[f_0]$, there exists a sequence of positive numbers $\{t_n\}_{n=1}^{\infty}$ such that $\lim_{n\rightarrow\infty}t_nf_n=f_0$. Now we need to prove that $\lim_{n\rightarrow\infty}t_n=0$.

Choosing a point $x_0$ in $\mathcal{T}(S)$, we have
\begin{equation*}
\lim_{n\rightarrow\infty}t_nl(f^{-1}_n(x_0),\cdot)=\lim_{n\rightarrow\infty}l(x_0,t_nf_n(\cdot))=l(x_0,f_0(\cdot)).
\end{equation*}
From the properly discontinuity of the action of $Mod(S)$ on $\mathcal{T}(S)$, we know that $f^{-1}_n(x_0)\rightarrow\infty$ in $\mathcal{T}(S)$. By the definition of the Thurston compactification, we know that $\lim_{n\rightarrow\infty}t_n l(f^{-1}_n(x_0),\cdot)=l(x_0,f_0(\cdot))=i(F,\cdot)$ for some $F$ in $\mathcal{MF}$. Then we have $\lim_{n\rightarrow\infty}t_n=0$.
\end{proof}

Now we give a description of the points in $\partial E$.
\begin{theo}\label{main2}
For any $[f]$ in $\partial E$, we have
\begin{equation*}
f(\cdot)=\sum_{i=1}^{m}i(E_i,\cdot)F_i,
\end{equation*}
where $E_i,F_i$ are some measured foliations with $i(E_i,E_j)=0$ and $i(F_i,F_j)=0$ for $i,j=1,2,...,m$.
\end{theo}
\begin{proof}
Since $[f]\in \partial E$, there exists a sequence $\{f_n\}_{n=1}^{\infty}$ in $Mod(S)$ such that $\lim_{n\rightarrow\infty}[f_n]=[f]$. By Lemma \ref{lemm7}, there exists a sequence of positive numbers $\{t_n\}_{n=1}^{\infty}$ such that $\lim_{n\rightarrow\infty}t_n=0$ and $\lim_{n\rightarrow\infty}t_nf_n=f$ in $\Omega(\mathcal{MF})$.

As $f$ is a map from $\mathcal{MF}$ to $\mathcal{MF}$, let $Imf=\{f(F):F\in\mathcal{MF}\}\subseteq\mathcal{MF}$ be the image of $f$. We claim that for any $F,G\in Imf$, $i(F,G)=0$. For any $F,G$ in $Imf$, there are $F_1,G_1\in\mathcal{MF}$ such that $f(F_1)=F$ and $f(G_1)=G$. Since $\lim_{n\rightarrow\infty}t_nf_n=f$ and $\lim_{n\rightarrow\infty}t_n=0$, we have
\begin{equation*}
i(F,G)=i(f(F_1),f(G_1))=\lim_{n\rightarrow\infty}i(t_nf_n(F_1),t_nf_n(G_1))=\lim_{n\rightarrow\infty}t_n^{2}i(F_1,G_1)=0.
\end{equation*}

From this claim, there are $m$ pairwise disjoint indecomposable measured foliations $\{F_i\}_{i=1}^{m}$ and some nonnegative function $f_i(\cdot)$ on $\mathcal{MF}$ such that
\begin{equation*}
f(\cdot)=\sum_{i=1}^{m}f_i(\cdot)F_i.
\end{equation*}

For $1\leq i\leq m$, from Lemma \ref{lemm0}, there exists a sequence of simple closed curves $\{\gamma_n\}_{n=1}^{\infty}$ such that
\begin{equation*}
i(\gamma_n,F_i)>0,\,\frac{i(\gamma_n,F_j)}{i(\gamma_n,F_i)}<\frac{1}{n}\,(j\neq i).
\end{equation*}
Then
\begin{equation*}
\frac{i(\gamma_n,f(\cdot))}{i(\gamma_n,F_i)}=f_i(\cdot)+\sum_{j\neq i}^{m}\frac{i(\gamma_n,F_j)}{i(\gamma_n,F_i)}f_j(\cdot)
\leq f_i(\cdot)+\frac{1}{n}\sum_{j\neq i}^{m}f_j(\cdot).
\end{equation*}
Set $G_n=\frac{1}{i(\gamma_n,F_i)}\gamma_n$. Then
\begin{equation*}
\lim_{n\rightarrow\infty}i(\overline{f}(G_n),\cdot)
=\lim_{n\rightarrow\infty}\frac{i(\gamma_n,f(\cdot))}{i(\gamma_n,F_i)}=f_i(\cdot).
\end{equation*}
In particular, for the filling curves $\{\alpha_j\}_{j=1}^{N}$ defined in Section \ref{sec2},
\begin{equation*}
\lim_{n\rightarrow\infty}i(\overline{f}(G_n),\alpha_j)=f_i(\alpha_j)\,\,(j=1,2,...,N).
\end{equation*}
Thus there exists a constant $M>0$ such that
\begin{equation*}
l(\overline{f}(G_n))=\sum_{j=1}^{N}i(\overline{f}(G_n),\alpha_j)\leq M.
\end{equation*}
From Lemma \ref{lemm1}, there exists a subsequence $\{G_{n_k}\}_{k=1}^{\infty}$ such that
\begin{equation*}
\lim_{k\rightarrow\infty}\overline{f}(G_{n_k})=E_i
\end{equation*}
for some $E_i$ in $\mathcal{MF}$. Then we have
\begin{equation*}
f_i(\cdot)=\lim_{k\rightarrow\infty}i(\overline{f}(G_{n_k}),\cdot)=i(E_i,\cdot).
\end{equation*}

For $1\leq i\leq m$, we construct a measured foliation $E_i$ as above. Since for any $F,G\in Im\overline{f}$, $i(F,G)=0$, we have $i(E_i,E_j)=0$ for $i\neq j$. Thus we have
\begin{equation*}
f=\sum_{i=1}^{m}i(E_i,\cdot)F_i,
\end{equation*}
and $E_i,F_i$ are measured foliations with $i(E_i,E_j)=0$ and $i(F_i,F_j)=0$ for $i,j=1,2,...,m$,
which completes the proof.
\end{proof}
\begin{rema} \label{rema3}
If $f(\cdot)=\sum_{i=1}^{m}i(E_i,\cdot)F_i$ as in Theorem \ref{main2}, then $\overline{f}(\cdot)=\sum_{i=1}^{m}i(F_i,\cdot)E_i$.
\end{rema}
\begin{prob}
Does the converse of Theorem \ref{main2} hold: for any $\{E_i\}_{i=1}^{m},\{F_i\}_{i=1}^{m}$ in $\mathcal{MF}$ with $i(E_i,E_j)=0$ and $i(F_i,F_j)=0$ ($i,j=1,2,...,m$), $[\sum_{i=1}^{m}i(E_i,\cdot)F_i]\in \partial E$?
\end{prob}

Now we construct some special points in $\partial E$. Firstly, we consider the limit of the sequence $\{f^n\}_{n=1}^{\infty}$ for some $f$ in $Mod(S)$. We need two results (see \cite{Ivanov}).
\begin{prop} \label{theo1}
Let $f=T_{\alpha_1}^{n_1}\circ T_{\alpha_2}^{n_2}\circ \cdot\cdot\cdot\circ T_{\alpha_k}^{n_k}$, where $\alpha_1,...,\alpha_k$ are pairwise disjoint simple closed curves, $T_{\alpha_i}$ is the Dehn Twist of $\alpha_i$ and $n_i\in Z$ $(i=1,2,...,k)$. Then for any $F\in\mathcal{MF}$, we have
\begin{equation*}
\lim_{n\rightarrow \pm\infty} \frac{f^{n}(F)}{|n|}=\sum_{i=1}^{k}|n_i|i(\alpha_i,F)\alpha_i.
\end{equation*}
\end{prop}

\begin{prop}\label{theo2}
Let $f\in Mod(S)$ be a pseudo-Anosov element such that $f(F^s)=\lambda^{-1}F^{s},\,f(F^u)=\lambda F^{u}$ with $\lambda>1,\,F^{s},F^{u}\in \mathcal{MF}$ and $i(F^s,F^u)=1$. Then for any $F\in\mathcal{MF}$, we have
\begin{equation*}
\lim_{n\rightarrow \infty} \frac{f^{n}(F)}{\lambda^{n}}=i(F^{s},F)F^{u},\,\,
\lim_{n\rightarrow \infty} \frac{f^{-n}(F)}{\lambda^{n}}=i(F^{u},F)F^{s}.
\end{equation*}
\end{prop}

From Proposition \ref{theo1} and Proposition \ref{theo2}, we have
\begin{prop} \label{prop5}
(1) With the assumption of Proposition \ref{theo1}, we have
\begin{equation*}
\lim_{n\rightarrow \pm\infty}[f^{n}(\cdot)]=[\sum_{i=1}^{k}|n_i|i(\alpha_i,\cdot)\alpha_i]\in \partial E.
\end{equation*}
(2) With the assumption of Proposition \ref{theo2}, we have
\begin{equation*}
\lim_{n\rightarrow \infty}[f^{n}(\cdot)]=[i(F^{s},\cdot)F^{u}]\in\partial E \,\,\,
\lim_{n\rightarrow \infty}[f^{-n}(\cdot)]=[i(F^{u},\cdot)F^{s}]\in\partial E.
\end{equation*}
\end{prop}

It is well known that the action of $Mod(S)$ on $\mathcal{PMF}$ is minimal, that is, the orbit of any element of $\mathcal{PMF}$ under the action of $Mod(S)$ is dense in $\mathcal{PMF}$ (see \cite{FLP}). We extend this result a little:
\begin{lemm} \label{lemm5}
Let $Mod^{'}(S)\subseteq Mod(S)$ be the set of all mapping classes preserving the punctures of $S$ pointwise. Then the action of $Mod^{'}(S)$ on $\mathcal{PMF}$ is minimal, that is, the orbit of any element of $\mathcal{PMF}$ under the action of $Mod^{'}(S)$ is dense in $\mathcal{PMF}$.
\end{lemm}
\begin{proof}
Since $\mathcal{S}$ is dense in $\mathcal{PMF}$, we only need to prove that for any $\alpha,\beta\in\mathcal{S}$, $\beta\in Cl(Mod^{'}(S)(\alpha))$, where $Cl(Mod^{'}(S)(\alpha))\subseteq\mathcal{PMF}$ is the closure of the orbit of $\alpha$ under the action of $Mod^{'}(S)$.

Take $\gamma\in\mathcal{S}$ such that $i(\alpha,\gamma)\neq 0$ and $i(\beta,\gamma)\neq 0$. Let $T_{\gamma}$ and $T_{\beta}$ be the Dehn Twist of $\gamma$ and $\beta$, respectively. Note that $T_{\gamma}$ and $T_{\beta}$ preserve each puncture of $S$. Thus $T_{\gamma},T_{\beta}\in Mod^{'}(S)$. By Theorem \ref{theo1}, $\lim_{n\rightarrow \infty} \frac{T_{\gamma}^{n}(\alpha)}{n}=i(\alpha,\gamma)\gamma$, which implies that $\gamma\in Cl(Mod^{'}(S)(\alpha))$. Using Theorem \ref{theo1} again, we have $\lim_{n\rightarrow \infty} \frac{T_{\beta}^{n}(\gamma)}{n}=i(\gamma,\beta)\beta$, which implies that $\beta\in Cl(Mod^{'}(S)(\gamma))$. Thus we have $\beta\in Cl(Mod^{'}(S)(\alpha))$.
\end{proof}

From Lemma \ref{lemm5}, we have
\begin{prop} \label{prop6}
For any $F,G\in \mathcal{MF}$, $[i(F,\cdot)G]\in \partial E$.
\end{prop}
\begin{proof}
From Lemma \ref{lemm5}, for a simple closed curve $\alpha$ in $S$, there are two sequences $\{f_n\}_{n=1}^{\infty}, \{g_n\}_{n=1}^{\infty}$ in $Mod^{'}(S)$ such that $\lim_{n\rightarrow\infty}[f_n(\alpha)]=[F]$ and $\lim_{n\rightarrow\infty}[g_n(\alpha)]=[G]$ in $\mathcal{PMF}$. Note that $[f_\alpha(\cdot)]=[i(\alpha,\cdot)\alpha]\in \partial E$ by Proposition \ref{prop5}(1). Then $\lim_{n\rightarrow\infty}[g_{n}\circ f_\alpha\circ f^{-1}_{n}(\cdot)]=\lim_{n\rightarrow\infty}[i(f_n(\alpha),\cdot)g_n(\alpha)]=[i(F,\cdot)G]\in \partial E$.
\end{proof}

We extend the result of Proposition \ref{prop6} by operation on subsurfaces. Let $\gamma_1,...,\gamma_p$ be disjoint essential simple closed curves in $S$. After cutting along these curves, we have some connected subsurfaces $S_1,...,S_k$. Then we have
\begin{prop} \label{prop7}
For any $a_i\geq0$ $(i=1,2,...,p)$, $F_j, G_j$ in $\mathcal{MF}(S_j)$ $(j=1,2,...,k)$,
\begin{equation*}
[\sum_{i=1}^{p}a_ii(\gamma_i,\cdot)\gamma_i+\sum_{j=1}^{k}i(F_j,\cdot)G_j]\in \partial E.
\end{equation*}
\end{prop}
\begin{proof}
For $j=1,2,...,k$, take a simple closed curve $\beta_{j}$ in $S_j$. Using Lemma \ref{lemm5} in each subsurface $S_j$ (seen as a punctured surface), we find two sequences $\{f_n\}_{n=1}^{\infty}, \{g_n\}_{n=1}^{\infty}$ in $Mod(S)$ such that $\lim_{n\rightarrow\infty}[f_n(\beta_j)]=[F_j]$, $\lim_{n\rightarrow\infty}[g_n(\beta_j)]=[G_j]$ for $j=1,2,...,k$ and $f_n(\gamma_i)=\gamma_i,g_n(\gamma_i)=\gamma_i$ for $i=1,2,...,p$. Thus for $j=1,2,...,k$, there are two sequences of positive numbers $\{t^{j}_{n}\}_{n=1}^{\infty}$, $\{s^{j}_n\}_{n=1}^{\infty}$ such that
\begin{equation*}
\lim_{n\rightarrow\infty}t^{j}_nf_n(\beta_j)=F_j,\,\,\,\lim_{n\rightarrow\infty}s^{j}_ng_n(\beta_j)=G_j\,(j=1,2,...,k).
\end{equation*}
From Proposition \ref{prop5}(1) and the denseness of the set of rational numbers in $R$, we have
\begin{equation*}
[h_n(\cdot)]=[\sum_{i=1}^{p}a_ii(\gamma_i,\cdot)+
\sum_{j=1}^{k}t^{j}_n s^{j}_n i(\beta_{j},\cdot)\beta_{j}]\in \partial E\,(n=1,2,...).
\end{equation*}
Thus
\begin{equation*}
[g_n\circ h_n\circ f^{-1}_n(\cdot)]=[\sum_{i=1}^{p}a_ii(\gamma_i,\cdot)+
\sum_{j=1}^{k} i(t^{j}_nf_n(\beta_{j}),\cdot)s^{j}_ng_n(\beta_{j})]\in \partial E\,(n=1,2,...).
\end{equation*}
Note that
\begin{equation*}
\lim_{n\rightarrow\infty}[g_n\circ h_n\circ f^{-1}_n(\cdot)]=[\sum_{i=1}^{p}a_ii(\gamma_i,\cdot)\gamma_i+\sum_{j=1}^{k}i(F_j,\cdot)G_j],
\end{equation*}
which implies that $[\sum_{i=1}^{p}a_ii(\gamma_i,\cdot)\gamma_i+\sum_{j=1}^{k}i(F_j,\cdot)G_j]\in \partial E$.
\end{proof}

\section{Some applications} \label{sec4}
Since $Mod(S)$ acts continuously on the Thurston compactification $\mathcal{T}^{Th}(S)=\mathcal{T}(S)\bigcup \mathcal{PMF}$ and the Gardiner-Masur compactification $\mathcal{T}^{GM}(S)=\mathcal{T}(S)\bigcup GM$ of $\mathcal{T}(S)$, we have two maps
$$\Pi_{Th}:Mod(S)\times \mathcal{T}^{Th}(S)\rightarrow\mathcal{T}^{Th}(S),\,(f,p)\mapsto f(p)$$
and $$\Pi_{GM}:Mod(S)\times \mathcal{T}^{GM}(S)\rightarrow\mathcal{T}^{GM}(S),\,(f,p)\mapsto f(p).$$
If we endow $Mod(S)$ with the discrete topology, then $\Pi_{Th}$ and $\Pi_{GM}$ are both continuous. Since $Cl(E)=E\bigcup \partial E$ is a completion of $Mod(S)$ in some sense, it may be natural to extend the domains of $\Pi_{Th}$ and $\Pi_{GM}$ to $Cl(E)\times \mathcal{T}^{Th}(S)$ and $Cl(E)\times \mathcal{T}^{GM}(S)$, respectively.

For this, we need equivalent models of $\mathcal{T}^{Th}(S)$ and $\mathcal{T}^{GM}(S)$. From the definitions of $\mathcal{T}^{Th}(S)$ and $\mathcal{T}^{GM}(S)$, a point in $\mathcal{T}^{Th}(S)$ and $\mathcal{T}^{GM}(S)$ are represented by $[p_1:\mathcal{S}\rightarrow R_{\geq 0}]$ and $[p_2:\mathcal{S}\rightarrow R_{\geq 0}]$, respectively, where $[p_i]\in PR_{\geq0}^{\mathcal{S}}$ is the projective class of $p_i\in R_{\geq0}^{\mathcal{S}}$. Since $R_{+}\times\mathcal{S}$ is dense in $\mathcal{MF}$, $p_1$ and $p_2$ extend to homogeneous continuous functions on $\mathcal{MF}$ (see \cite{Bonahon} and \cite{Miya-1}). Thus a point in $\mathcal{T}^{Th}(S)$ or $\mathcal{T}^{GM}(S)$ can be represented by the projective class of a homogeneous continuous function on $\mathcal{MF}$. Using these notations, the actions of $Mod(S)$ on $\mathcal{T}^{Th}(S)$ and $\mathcal{T}^{GM}(S)$ are defined as follows: for any $f\in Mod(S)$, $p_1=[p_1:\mathcal{MF}\rightarrow R_{\geq0}]\in\mathcal{T}^{Th}(S)$ and $p_2=[p_2:\mathcal{MF}\rightarrow R_{\geq0}]\in\mathcal{T}^{GM}(S)$, $f(p_1)=[p_1\circ f^{-1}]$ and $f(p_2)=[p_2\circ f^{-1}]$. Since the inverse operation $(\cdot)^{-1}$ on $Mod(S)$ extends to the conjugate operation $\overline{(\cdot)}$ on $Cl(E)$, we define the extensions of $\Pi_{Th}$ and $\Pi_{GM}$ as
\begin{theo}\label{main5}
Let $\Delta_1=\{([f],[p])\in Cl(E)\times \mathcal{T}^{Th}(S):p\circ\overline{f}(\cdot)\neq 0\}$ and $\Delta_2=\{([f],[p])\in Cl(E)\times \mathcal{T}^{GM}(S):p\circ\overline{f}(\cdot)\neq 0\}$. The two maps $\Psi_{Th}:\Delta_1\rightarrow \mathcal{T}^{Th}(S)$ and $\Psi_{GM}:\Delta_2\rightarrow \mathcal{T}^{GM}(S)$ defined by $\Psi_{Th}([f],[p])=[p\circ\overline{f}(\cdot)]$ and $\Psi_{GM}([f],[p])=[p\circ\overline{f}(\cdot)]$, respectively, are continuous.
\end{theo}
\begin{proof}
We only prove the continuity of $\Psi_{GM}$. The continuity of $\Psi_{Th}$ can be proved by a similar argument.

Suppose $\{[p_n]\}_{n=0}^{\infty}\subseteq \mathcal{T}^{GM}(S),\{[f_n]\}_{n=0}^{\infty}\subseteq Cl(E)$ and $\lim_{n\rightarrow\infty}[p_n]=[p_0]$, $\lim_{n\rightarrow\infty}[f_n]=[f_0]$. Up to some constants, we assume that $\lim_{n\rightarrow\infty}f_n=f_0$ in $\Omega(\mathcal{MF})$ and $p_n:\mathcal{MF}\rightarrow R_{\geq0}$ converges uniformly to $p_0:\mathcal{MF}\rightarrow R_{\geq0}$ on any compact subsets of $\mathcal{MF}$. By Proposition \ref{prop3}, $\lim_{n\rightarrow\infty}\overline{f_n}=\overline{f_0}$.

Observe that for any $F$ in $\mathcal{MF}$,
\begin{equation*}
|p_0\circ\overline{f_0}(F)-p_n\circ\overline{f_n}(F)|\leq
|p_0\circ\overline{f_0}(F)-p_0\circ\overline{f_n}(F)|+|p_0\circ\overline{f_n}(F)-p_n\circ\overline{f_n}(F)|.
\end{equation*}

Since $\lim_{n\rightarrow\infty}\overline{f_n}(F)=\overline{f_0}(F)$, we know that for any $\epsilon>0$, there exists $N_1>0$ such that for any $n>N_1$,
\begin{equation*}
|p_0\circ\overline{f_0}(F)-p_0\circ\overline{f_n}(F)|
<\frac{\epsilon}{2}
\end{equation*}
and $\{\overline{f_n}(F)\}_{n=1}^{\infty}\subseteq M$ for some compact subset $M$ of $\mathcal{MF}$.

Since $p_n(\cdot)$ converges uniformly to $p_0(\cdot)$ on compact set $M$, there exists $N_2>0$ such that for any $n>N_2$,
\begin{equation*}
|p_0\circ\overline{f_n}(F)-p_n\circ\overline{f_n}(F)|
<\frac{\epsilon}{2}.
\end{equation*}

Thus for any $n>\max\{N_1,N_2\}$,
\begin{equation*} |p_0\circ\overline{f_0}(F)-p_n\circ\overline{f_n}(F)|<\epsilon,
\end{equation*}
which implies that for any $F\in\mathcal{MF}$,
\begin{equation*}
\lim_{n\rightarrow\infty}p_n\circ\overline{f_n}(F)=
p_0\circ\overline{f_0}(F).
\end{equation*}
By the definition of $\mathcal{T}^{GM}(S)$,
\begin{equation*}
\Psi_{GM}([f_0],[p_0])=[p_0\circ\overline{f_0}(\cdot)]=
\lim_{n\rightarrow\infty}[p_n\circ\overline{f_n}(\cdot)]=
\lim_{n\rightarrow\infty}\Psi_{GM}([f_n],[p_n]),
\end{equation*}
which completes the proof.
\end{proof}

\begin{rema}\label{rema5}
(1) For $[f]\in Cl(E)$, $[p_1]\in \mathcal{T}^{Th}(S)$ and $[p_2]\in \mathcal{T}^{GM}(S)$, it may occur that $p_1\circ \overline{f}(\cdot)=0$ and $p_2\circ \overline{f}(\cdot)=0$, that is, the values of $p_1$, $p_2$ on the image of $\overline{f}$ are $0$. In these cases, $\Psi_{Th}$ and $\Psi_{GM}$ are degenerated at $([f],[p_1])$ and $([f],[p_2])$, respectively. Thus we restrict the definitions of $\Psi_{Th}$ and $\Psi_{GM}$ on $\Delta_1$ and $\Delta_2$, respectively.
\\(2) For $f_0\in Mod(S)$, $\Psi_{Th}([f_0],\cdot)$ and $\Psi_{GM}([f_0],\cdot)$ are defined on the whole $\mathcal{T}^{Th}(S)$ and $\mathcal{T}^{GM}(S)$, respectively. And $\Psi_{Th}([f_0],\cdot)$ and $\Psi_{GM}([f_0],\cdot)$ are consistent with the actions of $f_0$ on $\mathcal{T}^{Th}(S)$ and $\mathcal{T}^{GM}(S)$, respectively.
\\(3) For $x_0\in\mathcal{T}(S)$, $\Psi_{Th}(\cdot,x_0)$ and $\Psi_{GM}(\cdot,x_0)$ are both defined on the whole $Cl(E)$.
\end{rema}

By Theorem \ref{main5}, we have
\begin{coro}\label{coro4}
For any $x\in\mathcal{T}(S)$ and sequence $\{f_n\}_{n=1}^{\infty}\subseteq Mod(S)$, suppose that $\lim_{n\rightarrow \infty}[f_n]=[f_0]$ in $Cl(E)$ for some $[f_0]\in\partial E$, then $\lim_{n\rightarrow \infty}f_n(x)=[l(x,\overline{f_0}(\cdot)]\in\mathcal{PMF}$ in $\mathcal{T}^{Th}(S)$ and $\lim_{n\rightarrow \infty}f_{n}(x)=[Ext^{\frac{1}{2}}(x,\overline{f_0}(\cdot))]\in GM$ in $\mathcal{T}^{GM}(S)$.
\end{coro}
By Theorem \ref{main5} and Corollary \ref{coro4}, we answer Problem \ref{problem1}: for any $x_0$ in $\mathcal{T}(S)$, considering the orbit $\Gamma(x_0)$ of $x_0$ under the action of $Mod(S)$, how to describe the closure of $\Gamma (x_0)$ in $\mathcal{T}^{Th}(S)$ or $\mathcal{T}^{GM}(S)$? For this, we set
\begin{equation*}
\partial E^{Th}(x_0)=\{\Psi_{Th}([f],x_0):[f]\in \partial E\}\subseteq \mathcal{PMF};\,\,\partial E^{GM}(x_0)=\{\Psi_{GM}([f],x_0):[f]\in \partial E\}\subseteq GM.
\end{equation*}
Then we have
\begin{theo} \label{main3}
In $\mathcal{T}^{Th}(S)$, the closure of $\Gamma(x_0)$ is $\Gamma(x_0)\cup \partial E^{Th}(x_0)$. In $\mathcal{T}^{GM}(S)$, the closure of $\Gamma(x_0)$ is $\Gamma(x_0)\cup \partial E^{GM}(x_0)$. What's more, $\partial E^{Th}(x_0)=\mathcal{PMF}$.
\end{theo}
\begin{proof}
By Corollary \ref{coro4}, $\partial E^{Th}(x_0)$ is included in the closure of $\Gamma(x_0)$ in $\mathcal{T}^{Th}(S)$. Conversely, suppose $p\in\mathcal{T}^{Th}(S)$ is an element of the closure of $\Gamma(x_0)$ in $\mathcal{T}^{Th}(S)$ and $p\notin \Gamma(x_0)$. Then there exists a sequence $\{f_n\}_{n=1}^{\infty}\subseteq Mod(S)$ such that $\lim_{n\rightarrow\infty}f_n(x_0)=p$ in $\mathcal{T}^{Th}(S)$. Since $Mod(S)$ acts properly discontinuously on $\mathcal{T}(S)$, we know that $p\in \mathcal{PMF}$. By Theorem \ref{main} and Proposition \ref{prop1}, there exists a subsequence $\{f_{n_k}\}_{k=1}^{\infty}$ such that $\lim_{k\rightarrow\infty}[f_{n_k}]=[f_0]$ in $Cl(E)$ for some $[f_0]\in\partial E$. By Corollary \ref{coro4}, we have $$p=\lim_{k\rightarrow\infty}f_{n_k}(x_0)=\Psi_{Th}([f_0],x_0)\in \partial E^{Th}(x_0).$$
Thus the closure of $\Gamma(x_0)$ in $\mathcal{T}^{Th}(S)$ is $\Gamma(x_0)\cup \partial E^{Th}(x_0)$.

Using a similar argument, we know that the closure of $\Gamma(x_0)$ in $\mathcal{T}^{GM}(S)$ is $\Gamma(x_0)\cup \partial E^{GM}(x_0)$.

Now we prove $\partial E^{Th}(x_0)=\mathcal{PMF}$. From Proposition \ref{prop5}(1), we know that for any simple closed curve $\alpha$ in $S$, $[i(\alpha,\cdot)\alpha]\in \partial E$. Thus $[i(\alpha,\cdot)]=[l(x_0,i(\alpha,\cdot)\alpha)]\in \partial E^{Th}(x_0)$. Since the set of simple closed curves is dense in $\mathcal{PMF}$, we have $\partial E^{Th}(x_0)=\mathcal{PMF}$.
\end{proof}
\begin{rema}\label{rema8}
It is well-known that the action of $Mod(S)$ on $\mathcal{PMF}$ is minimal (see \cite{FLP}). This fact also implies $\partial E^{Th}(x_0)=\mathcal{PMF}$: since $Mod(S)$ acts properly discontinuously on $\mathcal{T}(S)$, $\partial E^{Th}(x_0)\cap \mathcal{PMF}\neq\emptyset$. By the minimal action of $Mod(S)$ on $\mathcal{PMF}$, this implies $\partial E^{Th}(x_0)=\mathcal{PMF}$.
\end{rema}
\begin{rema}\label{rema6}
The new boundary $\partial E$ is related to a special boundary of $Mod(S)$. Precisely, fixed a base point $x\in\mathcal{T}(S)$, sending $f\in Mod(S)$ to $f(x)\in \Gamma(x)\subseteq\mathcal{T}(S)$, we can identify $Mod(S)$ with the orbit $\Gamma(x)$ naturally. By Theorem \ref{main3}, the boundary of $\Gamma(x)$ in $\mathcal{T}^{Th}(S)$ is $\partial E^{Th}(x)=\mathcal{PMF}$. Thus $\partial E^{Th}(x)=\mathcal{PMF}$ can be seen as a boundary of $Mod(S)$. As a boundary of $Mod(S)$, $\partial E^{Th}(x)$ is homeomorphic to $\mathcal{PMF}$ but depends upon the base point $x$ heavily. Thus we get a family of boundaries $\{\partial E^{Th}(x):x\in\mathcal{T}(S)\}$ of $Mod(S)$ in which each boundary is isomorphic to $\mathcal{PMF}$. We may call each boundary $\partial E^{Th}(x)$ the Thurston boundary with base point $x$. By Theorem \ref{main5}, the new boundary $\partial E$ covers each boundary $\partial E^{Th}(x)$ in this family by a surjective continuous map $\Psi_{x}:\partial E\rightarrow \partial E^{Th}(x),\,\, [f]\mapsto \Psi_{Th}([f],x)$.
\end{rema}
\begin{rema}\label{rema7}
Different from the case of Thurston compactification, $\partial E^{GM}(x_0)$ may be not the whole boundary $GM$. From the compactness of $\partial E$, we only know that $\partial E^{GM}(x_0)$ is a compact subset of $GM$. And $\partial E^{GM}(x_0)$ contains some new points different from those known points in $GM$. A special kind of boundary point was constructed in \cite{Bourque}:
\begin{equation*}
[Ext^{\frac{1}{2}}\big(x_0,\sum_{i=1}^{k}n_ii(\alpha_i,\cdot)\alpha_i\big)],
\end{equation*}
where $\alpha_i$ are pairwise disjoint simple closed curves and $n_i>0$. By Proposition \ref{prop5}(1), $\partial E^{GM}(x_0)$ contains these points.
\end{rema}
Let $\mathcal{T}_{\epsilon}(S)=\{x\in\mathcal{T}(S):\underline{l}(x)>\epsilon\}$ be the $\epsilon-$Thick part of $\mathcal{T}(S)$, where $\underline{l}(x)=\min_{\alpha\in\mathcal{S}}l(x,\alpha)$. The following result characterizes the points of $\partial E^{GM}(x_0)$.
\begin{theo}\label{main4}
For any $p\in GM$, $p\in \partial E^{GM}(x_0)$ for some $x_0\in\mathcal{T}(S)$ if and only if there exists a sequence $\{p_n\}_{n=1}^{\infty}\subseteq \mathcal{T}_{\epsilon}(S)$ for some $\epsilon>0$ such that $\lim_{n\rightarrow\infty}p_n=p$.
\end{theo}
\begin{proof}
Suppose that $p\in \partial E^{GM}(x_0)$ for some $x_0\in\mathcal{T}(S)$. Then $p=\lim_{n\rightarrow\infty}f_n(x_0)$ for some sequence $\{f_n\}_{n=1}^{\infty}\subseteq Mod(S)$. Note that $\underline{l}\big(f_n(x_0)\big)\equiv\underline{l}(x_0)$. Set $\epsilon=\frac{1}{2}\underline{l}(x_0)$. Then $f_n(x_0)\in \mathcal{T}_{\epsilon}(S)$.

Suppose that $\{p_n\}_{n=1}^{\infty}\subseteq \mathcal{T}_{\epsilon}(S)$ for some $\epsilon>0$ and $\lim_{n\rightarrow\infty}p_n=p$. Then from the Mumford's compactness criterion, we know that after projecting to the moduli space $\mathcal{M}(S)=\mathcal{T}(S)/Mod(S)$, the sequence $\{p_n\}_{n=1}^{\infty}$ lies in a precompact set $\mathcal{M}_{\epsilon}(S)=\mathcal{T}_{\epsilon}(S)/Mod(S)$, which is the $\epsilon$-thick part of $\mathcal{M}(S)$. Thus passing to a subsequence, we assume that there exists a sequence $\{f_n\}_{n=1}^{\infty}$ in $Mod(S)$ and $x_0$ in $\mathcal{T}(S)$ such that
\begin{equation*}
\lim_{n\rightarrow\infty}f_n(p_n)=x_0.
\end{equation*}
Set $x_n=f_n(p_n)$. Then $p_n=f_n^{-1}(x_n)$ and $\lim_{n\rightarrow\infty}x_n=x_0$. By the compactness of $Cl(E)$, passing to a subsequence again, we assume that $[f_n^{-1}]$ converges to some $[f]$ in $Cl(E)$. Thus by Theorem \ref{main5},
\begin{equation*}
p=\lim_{n\rightarrow\infty}p_n=\lim_{n\rightarrow\infty}f_n^{-1}(x_n)=\lim_{n\rightarrow\infty}\Psi_{GM}([f^{-1}_n],x_n)=\Psi_{GM}([f],x_0)
=[Ext^{\frac{1}{2}}(x_0,\overline{f}(\cdot))].
\end{equation*}
Since $[Ext^{\frac{1}{2}}(x_0,\overline{f}(\cdot))]=p\in GM$, we have $[f]\in \partial E$. So $p\in \partial E^{GM}(x_0)$.
\end{proof}

\addcontentsline{toc}{section}{\refname}

\bibliographystyle{plain}
\bibliography{Anewboundaryofmcg}
\end{document}